\documentclass[a4paper]{amsart}

\usepackage{amsmath,amssymb,amsthm,amscd}
\usepackage[matrix,arrow,curve,tips]{xy}
            \SelectTips{cm}{}

\usepackage{mathrsfs}
\usepackage{stmaryrd}
\usepackage{mathabx}

\newtheorem{dfn}{Definition}[section]

\newtheorem{prop}[dfn]{Proposition}
\newtheorem{theo}[dfn]{Theorem}
\newtheorem{ex}[dfn]{Example}
\newtheorem{lem}[dfn]{Lemma}

\newtheorem*{theorem}{Theorem}

\newcommand{\RR}{\mathbb{R}}
\newcommand{\CC}{\mathbb{C}}

\newcommand{\NN}{\mathbb{N}}

\newcommand{\cD}{\mathcal{D}}
\newcommand{\cE}{\mathcal{E}}

\newcommand{\cM}{\mathcal{M}}

\newcommand{\Hom}{\text{Hom}}

\newcommand{\End}{\text{End}}

\newcommand{\fm}{\mathfrak{m}}

\newcommand{\fg}{\mathfrak{g}}

\newcommand{\com}{\mathbin{{\scriptstyle \circ }}}

\newcommand{\ev}{\mathord{\mathrm{ev}}}

\newcommand{\supp}{\mathord{\mathrm{supp}}}

\newcommand{\C}{\mathord{\mathcal{C}^{\infty}}}
\newcommand{\Cc}{\mathord{\mathcal{C}^{\infty}_{c}}}

\newcommand{\GGG}{\mathscr{G}^{\scriptscriptstyle\#}}

\title[]
      {Reflexivity of the space of transversal distributions}

\author{J. Kali\v{s}nik}
\address{Department of Mathematics, University of Ljubljana,
         Jadranska 19, 1000 Ljubljana, Slovenia;
         Institute of Mathematics, Physics and Mechanics,
         University of Ljubljana, Jadranska 19,
         1000 Ljubljana, Slovenia}
\email{jure.kalisnik@fmf.uni-lj.si}

\thanks{This research was supported by research grants J1-1690, N1-0137, N1-0237 and research program Analysis and Geometry P1-0291 from 
the Slovenian Research Agency ARRS}


\begin{document}

\begin{abstract}
We show that the Fr\'{e}chet space $\C(P)$ of smooth functions on the total space of a surjective submersion $\pi:P\to M$ is a
reflexive $\Cc(M)$-module.
\end{abstract}

\maketitle

\section{Introduction}

Spaces of distributions on a smooth manifold play an important role in the theory of linear partial differential equations.
The space $\cD'(P)$ of distributions on a manifold $P$ is defined as the continuous dual of the $LF$-space $\Cc(P)$
of smooth, compactly supported functions on $P$. In practical applications the space $\cD'(P)$ is usually too large, so we
often restrict ourselves to some subspace of it, such as for example Sobolev space $W^{k,p}(P)$ or some $L^{p}(P)$ space. In our
concrete case, we are
primarily interested in convolution of distributions, so it makes sense to study
the space of compactly supported distributions on $P$. It is denoted by $\cE'(P)$ and can be equivalently
described as the continuous dual of the Fr\'{e}chet space $\C(P)$ of smooth functions on $P$.

The notion of a distribution on a single manifold naturally extends
to that of a $\pi$-transversal distribution on the surjective submersion $\pi:P\to M$, which is by definition a continuous 
$\C(M)$-linear map from $\Cc(P)$ to $\C(M)$. A basic example of such a submersion is the trivial bundle
$t:M\times M\to M$ over a closed manifold $M$, where $t$ is the projection on the first factor. 
In this case we can identify the space $\cE'_{t}(M\times M)$ of $t$-transversal distributions with the space of
Schwartz kernels of continuous linear operators from $\C(M)$ to $\C(M)$ in the following way. If 
$D:\C(M)\to\C(M)$ is a continuous linear operator, we denote by $K_{D}$ its Schwartz kernel, an element of $\cD(M\times M)$.
The kernel $K_{D}$ is semiregular with respect
to $t$ (see \cite{Tre67}) and connected to $D$ by the formula
\[
D(f)(x)=\int_{M}K_{D}(x,y)f(y)dy,
\]
where the integral is meant in the distributional sense and $f\in\C(M)$. The kernel $K_{D}$ induces a continuous $\C(M)$-linear map
$T_{K_{D}}:\C(M\times M)\to\C(M)$, given by $T_{K_{D}}(F)=t_{*}(FK_{D})$ for $F\in\C(M\times M)$.
From the Schwartz kernels theorem it then follows that we have an isomorphism
\[
\cE'_{t}(M\times M)=\Hom_{\C(M)}(\C(M\times M),\C(M))\cong\End(\C(M)).
\]
This point of view gives us an algebra structure on the space $\cE'_{t}(M\times M)$, which corresponds to convolution of Schwartz kernels. However, there
is also the third point of view, which is connected to the isomorphism
\[
\cE'_{t}(M\times M)\cong\C(M,\cE'(M)).
\]
Here we view a $t$-transversal distribution $T\in\cE'_{t}(M\times M)$ as a smooth family $(T_{x})_{x\in M}$ 
of compactly supported distributions along the fibres of $t$.

The algebra structure on the space $\cE'_{t}(M\times M)$ is related to the fact that $M\times M$ is the pair Lie groupoid 
with the target map $t$. In \cite{LeMaVa17} the authors have constructed a convolution product on the space $\cE'_{t}(G)$ of
$t$-transversal compactly supported distributions on any Lie groupoid $G$. This construction extends the well known case of $G=M\times M$ and
provides a natural representation of $\cE'_{t}(G)$ into $\End(\C(M))$. Convolution algebra $\cE'_{t}(G)$ 
plays an important role in \cite{ErYu19}, where the authors used Lie groupoids to characterize classical pseudodifferential operators on manifolds.
In \cite{KaMr22} the authors constructed for any Lie groupoid $G$ associated bialgebroid $\Cc(\GGG,\fg)$ with a natural representation into
$\cE'_{t}(G)$, while in \cite{AnMoYu21} transversal distributions were used to
define the convolution algebra of Schwartz kernels on a singular foliation. Distributions transverse to a submersion were used also 
in \cite{AnSk11} to define a pseudodifferential calculus on a singular foliation.

We now turn to the case of an arbitrary surjective submersion $\pi:P\to M$. In this generality we do not have any natural product
on the space of compactly supported $\pi$-transversal distributions, defined by
\[
\cE'_{\pi}(P)=\Hom_{\Cc(M)}(\C(P),\Cc(M)).
\]
However, we can still consider $\pi$-transversal distributions as smooth families of distributions
along the fibres. In this paper we focus on the functional analytic aspects of the space $\cE'_{\pi}(P)$.
It is a standard result that the space $\C(P)$ is a reflexive Fr\'{e}chet space for any manifold $P$, with an explicit isomorphism 
$\hat{}:\C(P)\to\cE'(P)'$ given by $\hat{F}(v)=v(F)$ for $F\in\C(P)$ and $v\in\cE'(P)$.
Our main theorem extends this result to the setting of transversal distributions. 

\begin{theorem}
Let $\pi:P\to M$ be a surjective submersion. The map
\[
\hat{}:\C(P)\to\Hom_{\Cc(M)}(\cE'_{\pi}(P),\Cc(M))
\]
is an isomorphism of locally convex vector spaces. In particular, this implies that $\C(P)$ and $\cE'_{\pi}(P)$
are reflexive $\Cc(M)$-modules.
\end{theorem}

We will study applications of this result in a separate paper. More precisely, so far we have not used the fact that $\C(P)$ is an algebra
over $\Cc(M)$. The whole $\Cc(M)$-dual $\cE'_{\pi}(P)$ of $\C(P)$ is too large to be a coalgebra in the algebraic sense, so we need to
restrict ourselves to a suitable subspace of $\cE'_{\pi}(P)$. We hope to find a functorial construction of a coalgebra, assigned to any
smooth surjective submersion $\pi:P\to M$, which generalizes the construction in \cite{Mrc07} and contains enough information to reconstruct
the submersion $\pi:P\to M$.

\section{Preliminaries}

For the convenience of the reader and to fix the notations, we will first quickly review some
basic definitions concerning locally convex vector spaces, smooth manifolds, spaces of functions and spaces of distributions. 
See \cite{Tre67} and \cite{Hor67} for more details. All our locally convex vector spaces will be complex and Hausdorff
and all manifolds will be assumed to be smooth, Hausdorff and second countable.

A subset $B$ of a locally convex space $E$ is bounded if and only if the set $p(B)$ is a bounded subset of $\RR$ for any continuous
seminorm $p$ on $E$. We will denote by $E'=\Hom(E,\CC)$ the space of all continuous linear functionals on $E$. 
If $F$ is another locally convex space, we similarly define by $\Hom(E,F)$ the space of all continuous linear maps from $E$ to $F$. 
We will equip all these spaces of maps with the strong topology of uniform
convergence on bounded subsets. The basis of neighbourhoods of zero in $E'$ consists of sets of the form
\[
K(B,\epsilon)=\{v\in E'\,|\,|v(F)|<\epsilon\text{ for all }F\in B\},
\]
where $B$ is a bounded subset of $E$ and $\epsilon>0$. Similarly, the basis of neighbourhoods of zero 
in $\Hom(E,F)$ consists of sets of the form
\[
K(B,V)=\{T\in\Hom(E,F)\,|\,T(B)\subset V\},
\]
where $B$ is a bounded subset of $E$ and $V$ is a neighbourhood of zero in $F$. 
If $E$ and $F$ are modules over a complex algebra $A$, we will denote the corresponding space of continuous module 
homomorphisms by $\Hom_{A}(E,F)$. Since it is a subspace of $\Hom(E,F)$, it carries the induced topology.

We now recall the definition of the Fr\'{e}chet topology on $\C(M)$ for any smooth
manifold $M$. Choose local coordinates $\phi=(x_{1},\ldots,x_{l}):U\to\RR^{l}$ on $M$. For any $m\in\NN_{0}$ and any compact subset $K\subset U$
we define the seminorm $p_{K,m}$ on $\C(M)$ by
\[
p_{K,m}(F)=\sup_{\substack{x\in K,|\alpha|\leq m}}|D^{\alpha}(F)(x)|
\]
for $F\in\C(M)$. Here we denoted $D^{\alpha}(F)=\frac{\partial^{|\alpha|}F}{\partial x_{1}^{\alpha_{1}}\cdots\partial x_{l}^{\alpha_{l}}}$, where 
$\alpha=(\alpha_{1},\ldots,\alpha_{l})\in\NN_{0}^{l}$ is a multi-index and $|\alpha|=\alpha_{1}+\alpha_{2}+\ldots+\alpha_{l}$. To keep
notation simple, we avoided denoting the dependence of seminorm on the coordinate chart. 
The sets of the form
\[
V_{K,m,\epsilon}=\{F\in\C(M)\,|\,p_{K,m}(F)<\epsilon\}
\]
for $K$ compact subset in some chart of $M$, $m\in\NN_{0}$ and $\epsilon>0$ then form a subbasis of neighbourhoods of zero 
for the Fr\'{e}chet topology on $\C(M)$. With this
topology $\C(M)$ becomes a complete, metrizable locally convex vector space. This topology 
coincides with the topology of uniform convergence of all derivatives on compact subsets of $M$.

The subspace $\Cc(M)$ of $\C(M)$, consisting of compactly supported
functions, is not complete in the Fr\'{e}chet topology, which makes us consider a finer LF-topology on $\Cc(M)$. 
We first denote for any compact subset $K$ of $M$ by $\Cc(K)$ the space of functions with
support contained in $K$. The space $\Cc(K)$ is a closed subset of $\C(M)$ and hence a Fr\'{e}chet space itself.
The LF-topology on $\Cc(M)$ is then defined as the inductive limit topology with respect to the family of all subspaces
of the form $\Cc(K)$ for $K\subset M$ compact. In particular, an absolutely convex subset $V$ of $\Cc(M)$ is
a neighbourhood of zero in $\Cc(M)$ if and only if $V\cap\Cc(K)$ is a neighbourhood of zero in $\Cc(K)$ for every compact subset
$K$ of $M$. The space 
$\Cc(M)$ with LF-topology is a complete locally convex space, which is not metrizable, if $M$ is not compact.

The space of distributions on $M$ is defined as $\cD'(M)=\Cc(M)'$, while the space of compactly supported
distributions on $M$ is defined by $\cE'(M)=\C(M)'$. We have a natural inclusion of the space $\cE'(M)$
into $\cD'(M)$, whose image consists of distributions with compact support. The algebra $\C(M)$ acts
on $\cD'(M)$ by $(F\cdot v)(G)=v(FG)$ for $F\in\C(M)$, $G\in\Cc(M)$ and $v\in\cD'(M)$.
It is known that the space $\C(M)$ is reflexive, which
means that $\C(M)\cong\cE'(M)'$.

If $M$ is a smooth manifold and $E$ is a locally convex vector space,
a vector valued function $\alpha:M\to E$ is smooth if in local coordinates all partial derivatives exist
and are continuous. A map $\alpha:M\to E$ is scalarly smooth if the map $\phi\circ\alpha:M\to\CC$ is smooth for every $\phi\in E'$. If $\alpha$
is smooth it is also scalarly smooth, but the implication in the reverse direction does not always hold. However, 
if the space $E$ is complete, then every scalarly smooth function into $E$ is smooth \cite{KrMi97}.
We will denote by $\C(M,E)$ the space of smooth functions on $M$ with values in $E$ and by
$\Cc(M,E)$ its subspace, consisting of compactly supported functions. To make a distinction between scalar functions and vector valued functions, 
we will denote by $f(x)\in\CC$ the value of a function $f\in\C(M)$ at $x$ and by $u_{x}\in E$ the value of a function $u\in\C(M,E)$ at $x$.

\section{Correspondence between transversal distributions and families of distributions}

Let $P$ and $M$ be smooth manifolds and let $\pi:P\to M$ be a surjective submersion. We can view $P$ as a smooth family
of closed submanifolds $P_{x}=\pi^{-1}(x)$, parametrized by $x\in M$. Furthermore, smooth families of distributions on these fibres
can be equivalently described by $\pi$-transversal distributions.
In this section we will recall the definition of $\pi$-transversal
distributions on $P$ and its connection with smooth families of distributions along the fibres of $\pi$. This 
correspondence has been already used in \cite{AnSk11}, \cite{LeMaVa17}, \cite{ErYu19} and \cite{AnMoYu21}. 
However, since it plays a very important role in the proof of our main theorem and to make the presentation complete, we will
review the main definitions and constructions.

Note that the algebra $\C(M)$ acts on $\C(P)$ by $f\cdot F=(f\com\pi)F$ for $f\in\C(M)$ and $F\in\C(P)$. So we can
define the following vector spaces.

\begin{dfn}
Let $P$ and $M$ be smooth manifolds and suppose $\pi:P\to M$ is a surjective submersion. 
\begin{enumerate}
\item \textbf{The space of $\pi$-transversal distributions} on $P$ is the vector space
\[
\cD'_{\pi}(P)=\Hom_{\C(M)}(\Cc(P),\C(M)).
\]

\item \textbf{The space of $\pi$-transversal distributions with compact support} is defined similarly by 
\[
\cE'_{\pi}(P)=\Hom_{\Cc(M)}(\C(P),\Cc(M)).
\]

\item \textbf{A smooth family of distributions with compact supports along fibres of $\pi$} is an element $u=(u_{x})_{x\in M}\in\prod_{x\in M}\cE'(P_{x})$,
such that for every $F\in\C(P)$ the function $u(F)$, defined by $u(F)(x)=u_{x}(F|_{P_{x}})$, is smooth on $M$. The vector
space of all such smooth families will be denoted by
\[
\C(\prod_{x\in M}\cE'(P_{x})).
\]
\end{enumerate}
\end{dfn}

Let us take a look at some basic examples.

\begin{ex} \rm
(1) If $P$ is any smooth manifold and $M$ is a point, then the constant map $\pi:P\to M$ is a surjective submersion and in this
case the space of $\pi$-transversal distributions $\cE'_{\pi}(P)$ coincides with the space $\cE'(P)$ 
of ordinary compactly supported distributions on $P$.

(2) Let $\pi:P\to M$ be a surjective submersion and let $E$ be an image of a local section of $\pi$ 
(a submanifold of $P$ that maps by $\pi$ diffeomorphically to an open subset $U=\pi(E)$ of $M$). 
Denote by $\sigma_{E}:U\to E$ the smooth inverse of the map $\pi|_{E}$. For any $f\in\Cc(U)$ we define a $\pi$-transversal distribution 
$\llbracket E,f\rrbracket$ by
\[
\llbracket E,f\rrbracket(F)(x)=
\left\{ \begin{array} {ll}
      f(x)F(\sigma_{E}(x))   &;\hspace{1mm}x\in U,\\
      0                      &;\hspace{1mm}x\notin U,
       \end{array}\right.
\]
for $F\in\C(P)$. We think of $\llbracket E,f\rrbracket$ as a smooth family of Dirac distributions, supported on the section $E$ and weighted by
the function $f$. In particular, we have
\[
\llbracket E,f\rrbracket_{x}=f(x)\delta_{\sigma_{E}(x)}.
\] 
As an important special case of this construction let us consider the case when $M$ is closed, 
$P=M\times M$ and $t:M\times M\to M$ is the projection on the first factor. In this case
we can, by the Schwartz kernels theorem, identify $\cE'_{t}(M\times M)$ with the space $\Hom(\C(M),\C(M))$.
If $E\subset M\times M$ is a graph of a diffeomorphism $\Phi:M\to M$ and $f\equiv 1$ is the unit of $\C(M)$, then $\llbracket E,f\rrbracket$ 
corresponds to the Schwartz kernel of the pullback operator $\Phi^{*}:\C(M)\to\C(M)$.

(3) Suppose again that $M$ is a closed manifold, $P=M\times M$ and $t$ is projection on the first factor. For any linear partial differential
operator $D$ on $M$ we can define a $t$-transversal distribution $\llbracket M,D\rrbracket$ by
\[
\llbracket M,D\rrbracket(F)(x)=D(F|_{\{x\}\times M})(x).
\]
Here we have considered $F|_{\{x\}\times M}$ as a smooth function on $M$, so $D(F|_{\{x\}\times M})$ is well defined. This distribution is supported
on the diagonal $M\subset M\times M$ and computes the vertical $D$-derivative of $F$. It corresponds to the Schwartz kernel of the operator $D$.

More generally, let $\pi:P\to M$ be a surjective submersion and let $E$ be an image of a local section of $\pi$, as in the previous example.
Choose a linear partial differential operator $D$ on $P$ which acts along the fibres of $\pi$. For any $f\in\Cc(M)$ we then define 
$\llbracket E,fD\rrbracket\in\cE'_{\pi}(P)$ by
\[
\llbracket E,fD\rrbracket(F)(x)=
\left\{ \begin{array} {ll}
      f(x)D(F)(\sigma_{E}(x))   &;\hspace{1mm}x\in U,\\
      0                      &;\hspace{1mm}x\notin U.
       \end{array}\right.
\]

(4) Let $M=\RR^{n}$, $P=\RR^{n}\times\RR^{n}$ and let $\pi:\RR^{n}\times\RR^{n}\to\RR^{n}$ be the projection on the first factor. For  
$\phi\in\Cc(\RR^{n}\times\RR^{n})$ we define $\pi$-transversal distribution $T_{\phi}$ by
\[
T_{\phi}(F)(x)=\int_{\RR^{n}}\phi(x,y)F(x,y)dy.
\]
This $\pi$-transversal distribution corresponds to the family of smooth densities on $\RR^{n}$, parametrized by $\RR^{n}$. Explicitly,
\[
(T_{\phi})_{x}=\phi(x,-)dV,
\]
where $dV$ is Lebesgue measure on $\RR^{n}$. Distribution $T_{\phi}$ corresponds to the Schwartz kernel of the integral operator, associated to
the function $\phi$.
\end{ex}

The vector spaces $\cD'_{\pi}(P)$ and $\C(\prod_{x\in M}\cE'(P_{x}))$ are modules over $\C(P)$ and $\C(M)$, with
the actions given by:
\begin{align*}
(F\cdot T)(G)&=T(FG), \\
(f\cdot T)(G)&=T(f\cdot G), \\
(F\cdot u)_{x}&=F|_{P_{x}}\cdot u_{x}, \\
(f\cdot u)_{x}&=f(x)u_{x},
\end{align*}
for $T\in\cE'_{\pi}(P)$, $u\in\C(\prod_{x\in M}\cE'(P_{x}))$, $f\in\C(M)$, $G\in\Cc(P)$ and $F\in\C(P)$.

We will now recall the definition of the support of a transversal distribution. This will enable us to identify 
$\cE'_{\pi}(P)$ with the subspace of compactly supported distributions in $\cD'_{\pi}(P)$.

\begin{dfn}
\textbf{Support} of a $\pi$-transversal distribution $T\in\cD'_{\pi}(P)$ is the subset $\supp(T)$ of $P$, consisting of all points 
$p\in P$, which satisfy the condition that for every open neighbourhood $U$ of $p$ there exists $F\in\Cc(U)$ with $T(F)\neq 0$.
\end{dfn}
Note that we can equivalently define that $p\in\supp(T)^{c}$ if and only if there exists an open
neighbourhood $U$ of $p$ such that $T(F)=0$ for every $F\in\Cc(U)$. From this characterization it easily follows
that $\supp(T)$ is a closed subset of $P$.

The inclusions $\Cc(M)\hookrightarrow\C(M)$ and $\Cc(P)\hookrightarrow\C(P)$ are continuous, so we can define 
for any $T\in\cE'_{\pi}(P)$ the composition
\[
\Cc(P)\hookrightarrow\C(P)\overset{T}{\longrightarrow}\Cc(M)\hookrightarrow\C(M),
\]
which is continuous and $\C(M)$-linear. In particular, this composition defines an element of $\cD'_{\pi}(P)$, which
we will again denote by $T$. As a result we obtain an injective linear map $\cE'_{\pi}(P)\hookrightarrow\cD'_{\pi}(P)$,
which enables us to view $\cE'_{\pi}(P)$ as a vector subspace of $\cD'_{\pi}(P)$.
Similarly as in the case of distributions on a single manifold, one can prove the following basic result.

\begin{prop}\label{characterization of distributions with compact support}
A $\pi$-transversal distribution $T\in\cD'_{\pi}(P)$ belongs to $\cE'_{\pi}(P)$ if and only if $\supp(T)$ is compact.
\end{prop}

In the case of smooth families of distributions we have several different notions of support. If we consider a family of distributions
as a some kind of a vector valued map, it makes sense to define its support as a subset of the base manifold. On the other hand, in the view of 
the correspondence with transversal distributions, we will also consider its support as a subspace of the total space of the submersion.

\begin{dfn}
Let $\pi:P\to M$ be a surjective submersion and let $u$ be a smooth family of distributions with compact supports
along the fibres of $\pi$. 
\begin{enumerate}
\item The \textbf{support} of $u$ is the subset $\supp(u)$ of $M$, defined by 
\[
\supp(u)=\overline{\{x\in M\,|\,u_{x}\neq 0\}}.
\]
\item The \textbf{total support} of $u$ is the subset $\supp_{P}(u)$ of $P$, consisting of all points 
$p\in P$, which satisfy the condition that for every open neighbourhood $U$ of $p$ there exists $F\in\Cc(U)$ such that $u(F)\neq 0$.
\end{enumerate}
\end{dfn}

The relationship between these two types of support is described in the following proposition.
\begin{prop}
Let $\pi:P\to M$ be a surjective submersion.
\begin{enumerate}
\item For every smooth family $u\in\C(\prod_{x\in M}\cE'(P_{x}))$ we have
\[
\supp_{P}(u)=\overline{\bigcup_{x\in M}\supp(u_{x})}.
\]
\item Let $u\in\C(\prod_{x\in M}\cE'(P_{x}))$ be a smooth family of distributions. Then $\supp(u)$ is compact if and only if $\supp_{P}(u)$ is compact.
Moreover, if $\supp(u)$ and $\supp_{P}(u)$ are compact, we have $\pi(\supp_{P}(u))=\supp(u)$.
\end{enumerate}
\end{prop}
\begin{proof}
$(1)$ Let us denote $A=\overline{\bigcup_{x\in M}\supp(u_{x})}$. To show that $A\subset\supp_{P}(u)$, pick any $p\in A$ and arbitrary
neighbourhood $U$ of $p$. We can then find $x'\in\pi(U)$ and $p'\in U\cap\supp(u_{x'})$. From definition it follows that there exists
$F_{x'}\in\C(P_{x'})$ with compact support in $U\cap P_{x'}$ such that $u_{x'}(F_{x'})\neq 0$. If $F\in\Cc(P)$ is an extension of $F_{x'}$ with
$\supp(F)\subset U$, it follows that $u(F)(x')\neq 0$, which shows that $p\in\supp_{P}(u)$.

To prove the inverse inclusion, pick $p\in\supp_{P}(u)$ and any neighbourhood $U$ of $p$. Then there exists $F\in\Cc(U)$ such that
$u_{x}(F|_{P_{x}})\neq 0$ for some $x\in\pi(U)$. This implies that $U\cap\supp(u_{x})\neq\emptyset$ and hence $p\in A$.

$(2)$ By using contradiction we first prove that compactness of $\supp(u)$ implies compactness of $\supp_{P}(u)$. Suppose therefore that $\supp(u)$
is compact and that $\supp_{P}(u)$ is not compact. Then we can find a sequence of points $(p_{n})$ in $\supp_{P}(u)$, 
together with pairwise disjoint open neighbourhoods $U_{n}$ of $p_{n}$. 
Since $p_{n}\in\supp_{P}(u)$, there exists a function $F_{n}\in\Cc(U_{n})$ such that $u(F_{n})\neq 0$. The function $u(F_{n})$ lies in 
$\Cc(M)$ and by $\Cc(M)$-linearity of $u$ we
can furthermore assume that $u(F_{n})$ is nonnegative and that $\max(u(F_{n}))=n$.
Since the sets $U_{n}$ are pairwise disjoint, the function $F:P\to\CC$, defined by $F=\sum_{n=1}^{\infty}F_{n}$, is well defined and smooth, hence
$u(F)\in\Cc(M)$ by our assumption on compactness of $\supp(u)$. Now note that for any $x\in M$ the continuity of $u_{x}:\C(P_{x})\to\CC$ implies
\[
u(F)(x)=u_{x}(F|_{P_{x}})=u_{x}\left(\sum_{n=1}^{\infty}F_{n}|_{P_{x}}\right)=\sum_{n=1}^{\infty}u_{x}(F_{n}|_{P_{x}})
=\sum_{n=1}^{\infty}u(F_{n})(x),
\]
which shows that the sum $\sum_{n=1}^{\infty}u(F_{n})$ converges pointwise on $M$ to $u(F)$. Since this sum is unbounded on $M$,
this leads us into contradiction.

Suppose now that $\supp_{P}(u)$ is compact. We then have
\[
\pi(\supp_{P}(u))=\pi\left(\overline{\bigcup_{x\in M}\supp(u_{x})}\right)
=\overline{\pi\left(\bigcup_{x\in M}\supp(u_{x})\right)}=\supp(u)
\]
hence $\supp(u)$ is compact as well.
\end{proof}

We are thus led to define the space of compactly supported smooth families as follows.

\begin{dfn}
Let $\pi:P\to M$ be a surjective submersion.
A smooth family $u\in\C(\prod_{x\in M}\cE'(P_{x}))$ has \textbf{compact support} if $\supp(u)$ or equivalently $\supp_{P}(u)$ is
a compact set. The vector space of compactly supported smooth families will be denoted by $\Cc(\prod_{x\in M}\cE'(P_{x}))$.
\end{dfn}

In the remainder of this section we will show that we have a natural isomorphism between
vector spaces $\cE'_{\pi}(P)$ and $\Cc(\prod_{x\in M}\cE'(P_{x}))$.

For any closed submanifold $N$ of $P$ we denote by 
$I_{N}=\{F\in\C(P)\,|\,F|_{N}=0\}$ the ideal of functions that vanish on $N$. Furthermore,
for every $x\in M$ we denote by $\ev_{x}:\C(M)\to\CC$ the evaluation at the point $x$. The kernel of $\ev_{x}$ is
the maximal ideal $\fm_{x}=\{f\in\C(M)\,|\,f(x)=0\}$ of functions that vanish at $x$.
By using Taylor expansion it was observed in \cite{LeMaVa17} that any $F\in I_{P_{x}}$ can be written in the form
\[
F=f_{1}\cdot F_{1}+\ldots+f_{k}\cdot F_{k}
\]
for some functions $F_{1},\ldots,F_{k}\in\C(P)$ and $f_{1},\ldots,f_{k}\in\fm_{x}$. This gives us for every
$x\in M$ the equality
\[
\fm_{x}\cdot\C(P)=I_{P_{x}}.
\]
Now choose any $T\in\cE'_{\pi}(P)$ and any $F\in I_{P_{x}}$. From $\C(M)$-linearity of $T$ it follows 
\[
(\ev_{x}\circ T)(F)=T(f_{1}\cdot F_{1}+\ldots+f_{k}\cdot F_{k})(x)=f_{1}(x)T(F_{1})(x)+\ldots+f_{k}(x)T(F_{k})(x)=0,
\]
hence we obtain the induced continuous map $\ev_{x}\circ T:\C(P)/I_{P_{x}}\to\CC$. By composing this induced map 
with the natural isomorphism of locally convex spaces $\C(P)/I_{P_{x}}\cong\C(P_{x})$ we get for every $x\in M$ a continuous linear map
\[
T_{x}:\C(P_{x})\to\CC.
\]
Since $P_{x}$ is a closed submanifold of $P$, every function in $\C(P_{x})$ is a restriction of a function $F\in\C(P)$ and $T_{x}$
is then defined by $T_{x}(F|_{P_{x}})=T(F)(x)$, independent of the choice of $F$. It follows directly from the definition
that $(T_{x})_{x\in M}\in\prod_{x\in M}\cE'(P_{x})$ is a smooth compactly supported family and that there is a $\Cc(M)$-linear map
\[
\Phi:\cE'_{\pi}(P)\to\Cc(\prod_{x\in M}\cE'(P_{x})),
\]
given by
\[
\Phi(T)=(T_{x})_{x\in M}.
\]

To see that $\Phi$ is an isomorphism, we first consider a special case 
when $P=M\times N$ is a trivial bundle over $M$ with fiber $N$ and $\pi$ is the projection onto $M$.
In this case we can also define the vector space
\[
\Cc(M,\cE'(N))
\]
of all smooth compactly supported functions on $M$ with values in the locally convex space $\cE'(N)$. 
As shown in \cite{LeMaVa17} (see also \cite{Kal22}) we have the following basic result.

\begin{prop}\label{Isomorphism Phi for trivial bundles}
Let $P=M\times N$ and let $\pi:M\times N\to M$ be the projection. Then we have natural isomorphisms of vector spaces
\[
\cE'_{\pi}(P)\cong\Cc(\prod_{x\in M}\cE'(P_{x}))\cong\Cc(M,\cE'(N)).
\]
\end{prop}
The first isomorphism is given by the map $\Phi$, as described above, while for the second one we have to identify $\cE'(P_{x})=\cE'(\{x\}\times N)$
with $\cE'(N)$ for each $x\in M$ and then interpret a smooth family as a smooth $\cE'(N)$-valued map on $M$.

Now choose an open subset $U$ of $P$. We will denote by 
\[
\cE'_{\pi}(U)=\{T\in\cE'_{\pi}(P)\,|\,\supp(T)\subset U\}
\]
the subspace of $\cE'_{\pi}(P)$, consisting of distributions with support contained in $U$. We can also consider the restriction $\pi|_{U}:U\to\pi(U)$,
which is again a surjective submersion. We then have a continuous injective map 
\[
\cE'_{\pi|_{U}}(U)\hookrightarrow\cE'_{\pi}(P)
\]
with image $\cE'_{\pi}(U)$. In 
particular, vector spaces $\cE'_{\pi|_{U}}(U)$ and $\cE'_{\pi}(U)$ are isomorphic, but in general the topology on 
$\cE'_{\pi|_{U}}(U)$ is finer than the topology on $\cE'_{\pi}(U)$ induced from $\cE'_{\pi}(P)$.

\begin{theo}
For every surjective submersion $\pi:P\to M$, the map 
\[
\Phi:\cE'_{\pi}(P)\to\Cc(\prod_{x\in M}\cE'(P_{x})),
\]
defined by $\Phi(T)=(T_{x})_{x\in M}$, is an isomorphism of $\Cc(M)$-modules. 
\end{theo}
\begin{proof}
It follows directly from the definition that $\Phi$ is injective and $\Cc(M)$-linear, so we have to show that
it is surjective. Take any $u\in\Cc(\prod_{x\in M}\cE'(P_{x}))$. Since $\pi$ is a submersion and $\supp_{P}(u)\subset P$ is compact, we can cover $\supp_{P}(u)$ 
with open subsets $U_{1},U_{2},\ldots,U_{k}$ of $P$, such that $\pi$ restricted to each $U_{i}$ is a trivial fiber bundle. By using a suitable
partition of unity and $\C(P)$-module structure on $\Cc(\prod_{x\in M}\cE'(P_{x}))$, we can decompose 
\[
u=u_{1}+u_{2}+\ldots+u_{k},
\]
such that $\supp_{P}(u_{i})\subset U_{i}$ for each $i$. 
If we choose trivializations $U_{i}\approx \pi(U_{i})\times N_{i}$ for restrictions $\pi|_{U_{i}}$,
we can view $u_{i}$ as an element
of the space $\Cc(\pi(U_{i}),\cE'(N_{i}))$. Proposition \ref{Isomorphism Phi for trivial bundles} gives us a series of isomorphisms
\[
\Cc(\pi(U_{i}),\cE'(N_{i}))\cong\cE'_{\pi|_{U_{i}}}(U_{i})\cong\cE'_{\pi}(U_{i}).
\]
We then have $u=\Phi(T_{1}+\ldots+T_{k})$, where $T_{i}\in\cE'_{\pi}(U_{i})\subset\cE'_{\pi}(P)$ is the element, 
corresponding to $u_{i}$ under the above isomorphism.
\end{proof}

Let us mention that one can define a locally convex topology on the space $\Cc(\prod_{x\in M}\cE'(P_{x}))$, so that $\Phi$ becomes an isomorphism
of locally convex spaces (see \cite{LeMaVa17} for more details). However, in our case we will be interested 
in a particular description of topology on the space $\Cc(\RR^{l},\cE'(N))$.

We start by describing a basis of neighbourhoods of zero in $\Cc(\RR^{l})$ (see \cite{Hor67} for more details). 
Denote by $K_{0}=\emptyset\subset K_{1}\subset K_{2}\subset\ldots$ exhaustion of $\RR^{l}$ by balls $K_{n}=\{x\in\RR^{l}\,|\,|x|\leq n\}$ 
with centres at zero and radius $n\in\NN$. Furthermore, denote by $\textbf{m}=(m_{1},m_{2},\ldots)$
an increasing sequence of natural numbers and by $\textbf{e}=(\epsilon_{1},\epsilon_{2},\ldots)$ a decreasing
sequence of positive real numbers. The neighbourhood basis of zero for the LF-topology on $\Cc(\RR^{l})$
then consists of the sets of the form
\[
V_{\textbf{m},\textbf{e}}=\{f\in\Cc(\RR^{l})\,|\,|D^{\alpha}f(x)|<\epsilon_{n}\text{ for }x\in K_{n-1}^{c}\text{ and }|\alpha|\leq m_{n}\},
\]
as $\textbf{m}$ and $\textbf{e}$ vary over all sequences as above. In this spirit we also define a basis of neighbourhoods
of zero in $\Cc(\RR^{l},\cE'(N))$. For any bounded subset $B\subset\C(N)$ we define a seminorm $p_{B}$ on $\cE'(N)$ by 
\[
p_{B}(v)=\sup_{F\in B}|v(F)|
\] 
for $v\in\cE'(N)$. For any increasing sequence $\textbf{B}=(B_{1},B_{2},\ldots)$ of bounded subsets of $\C(N)$ and 
$\textbf{m}$ and $\textbf{e}$ as above we define
\[
V_{\textbf{B},\textbf{m},\textbf{e}}=\{u\in\Cc(\RR^{l},\cE'(N))\,|\,
p_{B_{n}}((D^{\alpha}u)_{x})<\epsilon_{n}\text{ for }x\in K_{n-1}^{c}\text{ and }|\alpha|\leq m_{n}\}.
\]
All such sets form a basis of neighbourhoods of zero for a locally convex topology on $\Cc(\RR^{l},\cE'(N))$ and we have
the following proposition.

\begin{prop}
Let $\pi:\RR^{l}\times N\to \RR^{l}$ be the projection onto $\RR^{l}$. The map 
\[
\Phi:\cE'_{\pi}(\RR^{l}\times N)\to\Cc(\RR^{l},\cE'(N))
\]
is then an isomorphism of locally convex spaces.
\end{prop}
\begin{proof}
We will use the isomorphism $\C(\RR^{l}\times N)\cong\C(\RR^{l},\C(N))$ under which
a function $F\in\C(\RR^{l}\times N)$ is identified with the smooth $\C(N)$-valued function on $\RR^{l}$, given by $x\mapsto F_{x}$,
where $F_{x}=F|\{x\}\times N$. This will enable us to compute partial derivatives of such functions in the directions on the base manifold.

First we show that $\Phi$ is continuous. Take any basic neighbourhood $V_{\textbf{B},\textbf{m},\textbf{e}}$ of zero
in $\Cc(\RR^{l},\cE'(N))$. We can then find a sequence $(\mu_{n})$ of positive real numbers, such that $\bigcup_{n=1}^{\infty}\mu_{n}B_{n}$
is a bounded subset of $\C(N)$. If we denote by $\pi_{N}:\RR^{l}\times N\to N$ the projection onto $N$, the set 
$B=\pi_{N}^{*}\left(\bigcup_{n=1}^{\infty}\mu_{n}B_{n}\right)$ is then a bounded subset of $\C(\RR^{l}\times N)$. Any function
in $B$ is of the form $F\com\pi_{N}$ for some $F\in\bigcup_{n=1}^{\infty}\mu_{n}B_{n}$ and hence corresponds to a
constant map $x\mapsto F$ in $\C(\RR^{l},\C(N))$. For any $u\in\Cc(\RR^{l},\cE'(N))$ and any multi-index $\alpha$ we thus have
\[
D^{\alpha}(u(F\com\pi_{N}))(x)=D^{\alpha}(x\mapsto u_{x}(F))(x)=(D^{\alpha}u)_{x}(F).
\]
Now define a sequence $\textbf{e}'$ by $\epsilon_{n}'=\frac{\mu_{n}\epsilon_{n}}{2}$ and take any $T\in K(B,V_{\textbf{m},\textbf{e}'})$.
For any $x\in K_{n-1}^{c}$ and any multi-index $\alpha$ with $|\alpha|\leq m_{n}$ we have
\begin{align*}
p_{B_{n}}((D^{\alpha}\Phi(T))_{x})&=\sup_{F\in B_{n}}|(D^{\alpha}\Phi(T))_{x}(F)|
=\sup_{F\in B_{n}}|D^{\alpha}(T(F\com\pi_{N}))(x)|, \\
&=\frac{1}{\mu_{n}}\sup_{F\in \mu_{n}B_{n}}|D^{\alpha}(T(F\com\pi_{N}))(x)|, \\
&\leq \frac{1}{\mu_{n}}\sup_{G\in B}|D^{\alpha}(T(G))(x)|.
\end{align*}
Since $G\in B$ and $T\in K(B,V_{\textbf{m},\textbf{e}'})$, it follows that $T(G)\in V_{\textbf{m},\textbf{e}'}$. So for 
$x\in K_{n-1}^{c}$ and any multi-index $\alpha$ with $|\alpha|\leq m_{n}$ we have $|D^{\alpha}(T(G))(x)|<\epsilon_{n}'$, which shows that
\[
p_{B_{n}}((D^{\alpha}\Phi(T))_{x})\leq \frac{1}{\mu_{n}}\sup_{G\in B}|D^{\alpha}(T(G))(x)|<\epsilon_{n}
\] 
and hence $\Phi(K(B,V_{\textbf{m},\textbf{e}'}))\in V_{\textbf{B},\textbf{m},\textbf{e}}$.

To see that $\Phi$ is an open map, take any basic neighbourhood $K(B,V_{\textbf{m},\textbf{e}})$ of zero
in $\cE'_{\pi}(\RR^{l}\times N)$. Now define a sequence $(B_{n})$ of bounded subsets of $\C(N)$ by
\[
B_{n}=\{(D^{\gamma}F)_{x}\,|\,x\in K_{n},|\gamma|\leq m_{n},F\in B\}.
\]
Furthermore, define a sequence $\textbf{e}'$ by $\epsilon_{n}'=\frac{\epsilon_{n}}{2^{m_{n}}}$. For any 
$u\in V_{\textbf{B},\textbf{m},\textbf{e}'}$, any $x\in K_{n-1}^{c}$, any $F\in B$ and any multi-index $\alpha$ 
with $|\alpha|\leq m_{n}$ we then have
\begin{align*}
|D^{\alpha}(u(F))(x)|&=|D^{\alpha}(x\mapsto u_{x}(F_{x}))(x)|
=\left|\sum_{\alpha=\beta+\gamma}(D^{\beta}u)_{x}(D^{\gamma}F)_{x}\right|, \\
&\leq \sum_{\alpha=\beta+\gamma}|(D^{\beta}u)_{x}(D^{\gamma}F)_{x}|.
\end{align*}
By definition we have $(D^{\gamma}F)_{x}\in B_{n'}$ for some $n'\geq n$, so we have a bound 
\[
|(D^{\beta}u)_{x}(D^{\gamma}F)_{x}|\leq p_{B_{n'}}((D^{\beta}u)_{x})<\epsilon_{n'}'<\epsilon_{n}'.
\]
Since there are at most $2^{m_{n}}$ possibilities of writing $\alpha=\beta+\gamma$, we get that
\[
|D^{\alpha}(u(F))(x)|<\epsilon_{n}.
\]
This shows that $V_{\textbf{B},\textbf{m},\textbf{e}'}\subset\Phi(K(B,V_{\textbf{m},\textbf{e}}))$ which concludes the proof.
\end{proof}

\section{Reflexivity of the space of transversal distributions}

The space of $\pi$-transversal distributions $\cE'_{\pi}(P)$ is the strong $\Cc(M)$-dual of the algebra of smooth functions $\C(P)$. In 
particular, $\cE'_{\pi}(P)$ is again a $\Cc(M)$-module, so we can consider its strong $\Cc(M)$-dual. We will show in this section
that we have a canonical isomorphism between $\C(P)$ and the strong $\Cc(M)$-dual of $\cE'_{\pi}(P)$.

Let $M$ be a smooth manifold and let $\cM$ be a $\Cc(M)$-module. For any $x\in M$ we define 'evaluation' of $\cM$ at $x$ as the quotient
\[
\cM(x)=\frac{\cM}{\fm_{x}\cdot\cM}.
\] 
We will first show that for any smooth surjective submersion $\pi:P\to M$ and any $x\in M$ 
the space $\cE'_{\pi}(P)(x)$ is isomorphic to $\cE'(P_{x})$ as a locally convex vector space. 

\begin{prop}\label{Localization of transversal distributions}
Let $\pi:P\to M$ be a surjective submersion.
\begin{enumerate}
\item For every $x\in M$ we have 
\[
\fm_{x}\cdot\cE'_{\pi}(P)=\{T\in\cE'_{\pi}(P)\,|\,T_{x}=0\}.
\]
\item The linear map $\Phi_{x}:\cE'_{\pi}(P)\to\cE'(P_{x})$, defined by $\Phi_{x}(T)=T_{x}$, induces
an isomorphism of vector spaces $\Phi_{x}:\cE'_{\pi}(P)(x)\to\cE'(P_{x})$ for every $x\in M$.
\end{enumerate}
\end{prop}
\begin{proof}
$(1)$ From definition it follows that $\fm_{x}\cdot\cE'_{\pi}(P)\subset\{T\in\cE'_{\pi}(P)\,|\,T_{x}=0\}$, 
so we only have to show the inclusion in the other direction. Choose
any $T\in\cE'_{\pi}(P)$ such that $T_{x}=0$. The set $K=\supp(T)\cap P_{x}$ is then a compact subset of $P_{x}$, so we can
find an open neighbourhood $U$ of $K$ in $P$, such that $\pi|_{U}$ is a trivial fiber bundle. Choose a partition
of unity $\{\rho_{U},\rho_{K^{c}}\}$, subordinated to the open cover $\{U,K^{c}\}$ of $P$. We can then decompose
\[
T=\rho_{U}T+\rho_{K^{c}}T.
\]
By construction, the support of $\rho_{K^{c}}T$ is compact and disjoint from $P_{x}$. So we can find $f_{K^{c}}\in\fm_{x}$ such that $f_{K^{c}}\equiv 1$
on a neighbourhood of $\pi(\supp(\rho_{K^{c}}T))$, hence 
\[
\rho_{K^{c}}T=f_{K^{c}}\cdot(\rho_{K^{c}}T)\in\fm_{x}\cdot\cE'_{\pi}(P).
\] 
On the other hand, let us choose a trivialization
$U\approx \pi(U)\times N$ for $\pi|_{U}$. We can then consider $\rho_{U}T$ as a smooth $\cE'(N)$-valued function on $\pi(U)$ with zero at $x$.
Taylor expansion (see \cite{KrMi97}) now gives us $T_{1},\ldots,T_{k}\in\Cc(\pi(U),\cE'(N))\cong\cE'_{\pi}(U)$ and
functions $f_{1},\ldots,f_{k}\in\fm_{x}$, such that
\[
\rho_{U}T=f_{1}\cdot T_{1}+\ldots+f_{k}\cdot T_{k}\in\fm_{x}\cdot\cE'_{\pi}(P).
\]

$(2)$ First we show that $\Phi_{x}:\cE'_{\pi}(P)\to\cE'(P_{x})$ is surjective. 
Choose any $v\in\cE'(P_{x})$ with $\supp(v)=K\subset P_{x}$. Similarly as above, let us choose an 
open neighbourhood $U\approx \pi(U)\times N$ of $K$ such that $\pi|_{U}$ is a trivial fiber bundle. Choose $f\in\Cc(\pi(U))$,
such that $f(x)=1$, and define $\overline{v}\in\Cc(\pi(U),\cE'(N))\cong\cE'_{\pi}(U)\subset\cE'_{\pi}(P)$ by 
\[
\overline{v}_{y}=f(y)v
\]
for $y\in\pi(U)$. We then have $\Phi_{x}(\overline{v})=v$. 

Since $\ker(\Phi_{x})=\{T\in\cE'_{\pi}(P)\,|\,T_{x}=0\}$, we get by $(1)$ the induced isomorphism
\[
\Phi_{x}:\cE'_{\pi}(P)(x)\to\cE'(P_{x}).
\]
\end{proof}
To show that the map $\Phi_{x}$ is an isomorphism of locally convex spaces, we need the following lemma.
Let $N$ be a smooth manifold and $K\subset N$ a compact subset. We then denote by $\cE'(K)$ 
the vector subspace of $\cE'(N)$ consisting of all distributions with support contained in $K$.

\begin{lem}\label{Limit structure on compactly supported distributions}
Let $N$ be a smooth manifold and let $V\subset\cE'(N)$ be an absolutely convex set. Then $V$ is a neighbourhood of zero in $\cE'(N)$
if and only if $V\cap\cE'(K)$ is a neighbourhood of zero in $\cE'(K)$ for each compact subset $K$ of $N$. 
\end{lem}
\begin{proof}
The space $\cE'(N)$ is barrelled and we have $\cE'(N)=\bigcup_{K}\cE'(K)$ as $K$ ranges over all compact subsets of $N$.
The proof now follows from Corollary $1.5$ in \cite{Val71}.
\end{proof}

\begin{prop}
Let $\pi:P\to M$ be a surjective submersion. The isomorphism
\[
\Phi_{x}:\cE'_{\pi}(P)(x)\to\cE'(P_{x})
\]
is an isomorphism of locally convex vector spaces for every $x\in M$.
\end{prop}
\begin{proof}
We have already shown that $\Phi_{x}$ is an isomorphism of vector spaces, so it remains to be
shown that it is continuous and open. It is actually enough to show that $\Phi_{x}:\cE'_{\pi}(P)\to\cE'(P_{x})$ is continuous and open.

To see that $\Phi_{x}$ is continuous, choose any basic open neighbourhood $K(B,\epsilon)$ of zero in $\cE'(P_{x})$,
where $B\subset\C(P_{x})$ is bounded and $\epsilon>0$. Note that the restriction map $\text{Res}_{x}:\C(P)\to\C(P_{x})$ is surjective and
continuous. Using a tubular neighbourhood of $P_{x}$ in $P$ we can construct a continuous section
$\text{Ext}_{x}:\C(P_{x})\to\C(P)$ of $\text{Res}_{x}$. It follows that $\text{Ext}_{x}(B)$ is bounded in $\C(P)$ and
\[
\Phi_{x}(K(\text{Ext}_{x}(B),V))\subset K(B,\epsilon),
\]
where $V=\{f\in\Cc(M)\,|\,|f(x)|<\epsilon\}$ is an open neighbourhood of zero in $\Cc(M)$.

We will now show that $\Phi_{x}$ is an open map. First we consider the case when $P=\RR^{l}\times N$ is a trivial bundle. 
It is enough to show that for any basic open neighbourhood $V_{\textbf{B},\textbf{m},\textbf{e}}$ of zero in $\Cc(\RR^{l},\cE'(N))$
the set $\Phi_{x}(V_{\textbf{B},\textbf{m},\textbf{e}})$ is a neighbourhood of zero in $\cE'(P_{x})\cong\cE'(N)$.
Take such $n\in\NN$ that $x\in\text{Int}(K_{n})$ and then choose $\rho\in\Cc(\text{Int}(K_{n}))$
with $\rho(x)=1$ and denote
$M=\max_{y\in K_{n},|\alpha|\leq m_{n}}|(D^{\alpha}\rho)(y)|$. Now define $\overline{v}=\rho v\in\Cc(\RR^{l},\cE'(N))$ for any $v\in\cE'(N)$. 
If $v\in K(B_{n},\frac{\epsilon_{n}}{M})$, we have for
any $y\in K_{n}$ and any $|\alpha|\leq m_{n}$ the following bound
\[
p_{B_{n}}(D^{\alpha}(\overline{v})_{y})=p_{B_{n}}((D^{\alpha}\rho)(y)v)\leq Mp_{B_{n}}(v)<\epsilon_{n},
\]
which shows that $\overline{v}\in V_{\textbf{B},\textbf{m},\textbf{e}}$ and
hence $K(B_{n},\frac{\epsilon_{n}}{M})\subset\Phi_{x}(V_{\textbf{B},\textbf{m},\textbf{e}})$.

In the case of a general submersion $\pi:P\to M$, choose an absolutely convex neighbourhood $V$ of zero in $\cE'_{\pi}(P)$. We need to show that
$\Phi_{x}(V)\cap\cE'(K)$ is a neighbourhood of zero in $\cE'(K)$ for every compact subset $K\subset P_{x}$. 
First choose an open neighbourhood $U\approx\pi(U)\times N$ of $K$ in $P$, on
which $\pi$ is trivial and $\pi(U)\approx\RR^{l}$. The set $V\cap\cE'_{\pi}(U)$ is then a neighbourhood of zero in $\cE'_{\pi}(U)\cong\Cc(\pi(U),\cE'(N))$,
which by previous paragraph implies that $\Phi_{x}(V\cap\cE'_{\pi}(U))$ is a neighbourhood of zero in $\cE'(N)$. Since
$\cE'(K)\subset\cE'(N)\subset\cE'(P_{x})$, this also implies that $\Phi_{x}(V)\cap\cE'(K)$ is a neighbourhood of zero in $\cE'(K)$. The proof now follows from
Lemma \ref{Limit structure on compactly supported distributions}.
\end{proof}

In the remainder of this section we will compute the strong $\Cc(M)$-dual of the space $\cE'_{\pi}(P)$. It is the locally convex space
\[
\Hom_{\Cc(M)}(\cE'_{\pi}(P),\Cc(M)),
\]
equipped with the topology of uniform convergence on 
bounded subsets of $\cE'_{\pi}(P)$. For every $F\in\C(P)$ we can define an element 
$\hat{F}\in\Hom_{\Cc(M)}(\cE'_{\pi}(P),\Cc(M))$ by $\hat{F}(T)=T(F)$ for $T\in\cE'_{\pi}(P)$, to
obtain a $\Cc(M)$-linear map
\[
\hat{}:\C(P)\to\Hom_{\Cc(M)}(\cE'_{\pi}(P),\Cc(M)).
\]
Our goal will be to show that the above map is an isomorphism of locally convex vector spaces.

Choose any $a\in\Hom_{\Cc(M)}(\cE'_{\pi}(P),\Cc(M))$, any $x\in M$ and any $T\in\fm_{x}\cdot\cE'_{\pi}(P)$, 
which we write as $T=f_{1}\cdot T_{1}+\ldots+f_{k}\cdot T_{k}$ for some $f_{i}\in\fm_{x}$ and $T_{i}\in\cE'_{\pi}(P)$. We now have
\[
a(T)(x)=a(f_{1}\cdot T_{1}+\ldots+f_{k}\cdot T_{k})(x)=f_{1}(x)a(T_{1})(x)+\ldots+f_{k}(x)a(T_{k})(x)=0,
\]
which implies that we have the induced continuous map
\[
a_{x}:\cE'_{\pi}(P)(x)\to\CC.
\]
Since $\cE'_{\pi}(P)(x)\cong\cE'(P_{x})$ by Proposition \ref{Localization of transversal distributions} and since $\cE'(P_{x})'\cong\C(P_{x})$, we can identify 
$a_{x}$ with some function $\check{a}_{x}\in\C(P_{x})$ such that $(\check{a}_{x})^{\hat{}}=a_{x}$. We can therefore assign
to $a$ a family of smooth functions $(\check{a}_{x})_{x\in M}\in\prod_{x\in M}\C(P_{x})$, which can be seen as
a function $\check{a}:P\to\CC$, which is smooth along the fibres of $\pi$. 

To show that $\check{a}$ is a smooth function on $P$, we first recall that $\C(P)$ acts on $\cE'_{\pi}(P)$ by $(F\cdot T)(G)=T(FG)$ 
for $F,G\in\C(P)$ and $T\in\cE'_{\pi}(P)$.
Since the multiplication map $F\cdot\underline{\hspace{2mm}}:\cE'_{\pi}(P)\to\cE'_{\pi}(P)$ is continuous,
we get the induced action of $\C(P)$ on $\Hom_{\Cc(M)}(\cE'_{\pi}(P),\Cc(M))$, defined by
\[
(F\cdot a)(T)=a(F\cdot T),
\]
for $F\in\C(P)$, $a\in\Hom_{\Cc(M)}(\cE'_{\pi}(P),\Cc(M))$ and $T\in\cE'_{\pi}(P)$. We have the basic relation
\[
\widecheck{F\cdot a}=F\check{a},
\]
which shows that the operation $a\mapsto\check{a}$ is $\C(P)$-linear.

\begin{prop}
Let $\pi:P\to M$ be a surjective submersion. Then for every $a\in\Hom_{\Cc(M)}(\cE'_{\pi}(P),\Cc(M))$ the induced function
$\check{a}:P\to\CC$ is smooth and we have $(\check{a})^{\hat{}}=a$.
\end{prop}
\begin{proof}
Assume first that $P=M\times N$ is a trivial bundle over $M$ with fiber $N$. To show the main idea, we will for simplicity also 
assume that $M$ is compact. We can then view the family $\check{a}=(\check{a}_{x})_{x\in M}$ as a function on $M$
with values in the Fr\'{e}chet space $\C(N)$. To show that it is smooth, it is enough to show that it is scalarly smooth since $\C(N)$ is complete 
(see Theorem $2.14$ in \cite{KrMi97}). Denote for any $v\in\cE'(N)$ by $\overline{v}\in\C(M,\cE'(N))$ the constant transversal distribution with value $v$. 
The function $x\mapsto \check{a}_{x}$ is then scalarly smooth on $M$ since
\[
x\mapsto \check{a}_{x}(v)=a(\overline{v})(x)
\]
and $a(\overline{v})\in\C(M)$. This shows that $\check{a}:M\to\C(N)$ is smooth, hence $\check{a}\in\C(P)$. If $M$ is not compact,
one can show by using partitions of unity that $\check{a}$ is locally smooth, which however implies smoothness. 

Now we consider the case of a general submersion $\pi:P\to M$. Choose any $p\in P$ and an open neighbourhood $U\approx \pi(U)\times N$ of $p$, on
which $\pi$ is trivial. Furthermore, choose $\chi\in\Cc(U)$ which is equal to $1$ on some neighbourhood of $p$ 
and suppose $\supp(\chi)$ is contained in $L\times K$, where $L\subset\pi(U)$ and $K\subset N$ are compact subsets. The element
$\chi\cdot a\in\Hom_{\Cc(M)}(\cE'_{\pi}(P),\Cc(M))$ then maps into $\Cc(L)$ and induces a continuous mapping
\[
\cE'_{\pi|_{U}}(U)\hookrightarrow\cE'_{\pi}(P)\overset{\chi\cdot a}{\longrightarrow}\Cc(L)\hookrightarrow\Cc(\pi(U)).
\]
By the above paragraph $\widecheck{\chi\cdot a}$ is a smooth function on $U$. But since $\widecheck{\chi\cdot a}=\chi\check{a}$ and since $\chi\equiv 1$
on some neighbourhood of $p$, this implies that $\check{a}$ is smooth on some neighbourhood of $p$. 
The equality $(\check{a})^{\hat{}}=a$ then follows directly from the definition.
\end{proof}

We are now ready to state and prove our main theorem.

\begin{theo}
Let $\pi:P\to M$ be a surjective submersion. The map
\[
\hat{}:\C(P)\to\Hom_{\Cc(M)}(\cE'_{\pi}(P),\Cc(M))
\]
is an isomorphism of locally convex vector spaces. In particular, this implies that $\C(P)$ and $\cE'_{\pi}(P)$
are reflexive $\Cc(M)$-modules.
\end{theo}
\begin{proof}
We first show that the given map is an isomorphism of vector spaces. To show that it is injective, 
suppose $\hat{F}=\hat{G}$ for some $F,G\in\C(P)$. This implies
that $T(F)=T(G)$ for every $T\in\cE'_{\pi}(P)$. In particular, for every $x\in M$ we have $T_{x}(F_{x})=T_{x}(G_{x})$.
Since elements of $\cE'(P_{x})$ separate points of $\C(P_{x})$, we have that $F_{x}=G_{x}$ for every $x\in M$ and hence $F=G$.
The surjectivity follows from the fact that for all $a\in\Hom_{\Cc(M)}(\cE'_{\pi}(P),\Cc(M))$ we have $(\check{a})^{\hat{}}=a$.  

Next we show that $\hat{}:\C(P)\to\Hom_{\Cc(M)}(\cE'_{\pi}(P),\Cc(M))$ is continuous.
Choose an arbitrary basic neighbourhood
$K(B,V)$ of zero in $\Hom_{\Cc(M)}(\cE'_{\pi}(P),\Cc(M))$. The subset $B$ of $\cE'_{\pi}(P)$ is bounded, hence it
is, by the Banach-Steinhaus theorem, equicontinuous . So we can find a neighbourhood $V'$ of zero in $\C(P)$,
such that $B(V')\subset V$. For every $F\in V'$ and every $T\in B$ it then follows $\hat{F}(T)=T(F)\in V$, which shows that
$\widehat{V'}\subset K(B,V)$.

Finally, we have to show that the given map is open. We can choose a subbasis of open neighbourhoods of zero in $\C(P)$
in the following way. Choose an open subset $U\approx\pi(U)\times\RR^{k}$ of $P$ on which $\pi$ is a trivial bundle with fiber $\RR^{k}$ and
assume $\pi(U)\approx\RR^{l}$. Choose compact subsets $L\subset\pi(U)$ and $K\subset\RR^{k}$, $n\in\NN_{0}$, $\epsilon>0$ and define
\[
V_{L\times K,n,\epsilon}=\{F\in\C(P)\,|\,|D^{\alpha}_{x}D^{\beta}_{y}F(x,y)|<\epsilon\text{ for }x\in L,\,y\in K,\,|\alpha+\beta|\leq n\}.
\]
Every neighbourhood of zero in $\C(P)$ contains a finite intersection of sets as above, so it is enough to show that 
the image of $V_{L\times K,n,\epsilon}$ is open. To this extent choose $\eta\in\Cc(\pi(U))$, such that $\eta\equiv 1$
on some neighbourhood of $L$ and define a subset $B$ of $\cE'_{\pi}(P)$ by
\[
B=\{\llbracket E_{y},\eta\,D_{y}^{\beta}\rrbracket\,|\,y\in K,|\beta|\leq n\},
\]
where $E_{y}$ is the image of the constant section $\sigma_{E_{y}}:\pi(U)\to\pi(U)\times\RR^{k}$ with value $y$. To see that $B$ is bounded
in $\cE'_{\pi}(P)$, it is by the Banach-Steinhaus theorem enough to show that $B(F)$ is bounded in $\Cc(M)$ for every $F\in\C(P)$.
Note that $B(F)$ lies in the Fr\'{e}chet space $\Cc(\supp(\eta))$, so we have to show that $p_{\supp(\eta),m}(B(F))$ is bounded in $\RR$
for every $m\in\NN_{0}$. This follows from
\begin{align*}
p_{\supp(\eta),m}(B(F))&=\sup_{\substack{T\in B\\ x\in\supp(\eta) \\ |\alpha|\leq m}}|D^{\alpha}_{x}(T(F))(x)|
=\sup_{\substack{(x,y)\in\supp(\eta)\times K \\ |\alpha|,|\beta|\leq m}}|D^{\alpha}_{x}D^{\beta}_{y}(\eta\cdot F)(x,y)|, \\
&\leq p_{\supp(\eta)\times K,2m}(\eta\cdot F).
\end{align*}
Now denote by $V_{L,n,\frac{\epsilon}{2}}$ the neighbourhood of zero in $\Cc(M)$, given by
\[
V_{L,n,\frac{\epsilon}{2}}=\{f\in\Cc(M)\,|\,|D^{\alpha}f(x)|<\tfrac{\epsilon}{2}\text{ for }x\in L,|\alpha|\leq n\}.
\]
For any $a\in K(B,V_{L,n,\frac{\epsilon}{2}})$ we then have
\[
\sup_{\substack{T\in B\\ x\in L \\ |\alpha|\leq n}}|D^{\alpha}_{x}(T(\check{a}))(x)|
=\sup_{\substack{T\in B\\ x\in L \\ |\alpha|\leq n}}|D^{\alpha}_{x}(a(T))(x)|<\epsilon
\]
and consequently
\[
\sup_{\substack{(x,y)\in L\times K \\ |\alpha+\beta|\leq n}}|D^{\alpha}_{x}D^{\beta}_{y}\check{a}(x,y)|
\leq \sup_{\substack{(x,y)\in L\times K \\ |\alpha|,|\beta|\leq n}}|D^{\alpha}_{x}D^{\beta}_{y}\check{a}(x,y)|
=\sup_{\substack{T\in B\\ x\in L \\ |\alpha|\leq n}}|D^{\alpha}_{x}(T(\check{a}))(x)|<\epsilon.
\]
This shows that $\check{a}\in V_{L\times K,n,\epsilon}$ and hence $K(B,V_{L,n,\frac{\epsilon}{2}})\subset \widehat{V_{L\times K,n,\epsilon}}$.
\end{proof}

\noindent
{\bf Acknowledgements:} 
I would like to thank O. Dragičević, F. Forstnerič, A. Kostenko and J. Mrčun for support during research.

\end{document}